\documentclass[12pt]{article}
\usepackage{amsmath,amssymb,extarrows}
\usepackage{hyperref}
\usepackage{url}
\usepackage{enumerate}
\usepackage{tabularx}
\usepackage{lscape}
\usepackage{longtable}
\usepackage{amsfonts,amsthm,cite,color}
\usepackage{epsfig}
\allowdisplaybreaks[4]

\newtheorem{theorem}{Theorem}[section]
\newtheorem{lemma}[theorem]{Lemma}
\newtheorem{corollary}[theorem]{Corollary}
\theoremstyle{definition}

\theoremstyle{remark}

\numberwithin{equation}{section}

\begin{document}

\begin{center}
		{\Large \bf {Double-sum Rogers--Ramanujan type identities}}
		\vskip 6mm
		{Duanyu Chen$^a$,  Xiangxin Liu$^b$ and Lisa Hui Sun$^{c*}$
			\\[2mm]
			Center for Combinatorics, LPMC, Nankai University, Tianjin 300071, P.R. China \\[2mm]
$^a$chen.duanyu@mail.nankai.edu.cn, $^b$liuxx@mail.nankai.edu.cn, $^c$sunhui@nankai.edu.cn \\[2mm]}
	\end{center}



{\noindent \bf Abstract.}
As the $q$-analog of Chebyshev polynomials,   $q$-Hermite polynomials form a cornerstone in the family of $q$-orthogonal polynomials, which play a fundamental role in quantum algebra and mathematical physics. Recently, Andrews obtained a series of Rogers--Ramanujan type identities by constructing Bailey pairs from Chebyshev polynomials. In this paper, by applying the expansion formula of  Chebyshev polynomials in terms of $q$-Hermite polynomials and using the orthogonality relations, we derive a series of Rogers--Ramanujan type identities on double sums, which further generalized the known results due to Andrews, Shi, Sun and Yao.

	{\noindent \bf Keywords:}  Chebyshev polynomials; $q$-Hermite polynomials; Rogers-Ramanujan type identities
	
	{\noindent \bf AMS Classification:}  05A30; 33D45
	
	\allowdisplaybreaks

\section{Introduction}

The famous Rogers--Ramanujan identities
\begin{align}
&\sum_{n=0}^{\infty} \frac{q^{n^2}}{(q; q)_n} =\frac{1}{(q, q^4; q^5)_\infty}, \label{RR1}\\
&\sum_{n=0}^{\infty} \frac{q^{n^2+n}}{(q; q)_n} =\frac{1}{(q^2, q^3; q^5)_\infty}, \label{RR2}
\end{align}
were first discovered and proved in 1894 by Rogers  \cite{Rog94}, and were rediscovered independently by Ramanujan \cite{RR19}, which have  profound applications in combinatorics, number theory, modular forms, statistical mechanics and representation theory. Thereby the series--product identities similar to the above form are called Rogers--Ramanujan type identities, which were systematically studied and generalized by Bailey \cite{Bai47,Bai49}. Later, Slater \cite{Sla51,Sla52} compiled and discovered a list of 130 identities of Rogers-Ramanujan type applying various Bailey pairs, which is widely known as ``Slater's list". Sills, Laughlin and Zimmer \cite{Sil03,LSZ08} further generalized Slater's results and obtained a more complete list of Rogers--Ramanujan type identities.

MacMahon \cite{Mac16}  and Schur \cite{Sch17} also
gave the combinatorial versions of the two Rogers--Ramanujan identities.  In 1961, Gordon \cite{Gor61} provided a combinatorial interpretation of the generalized Rogers--Ramanujan identities. Then in 1974, Andrews \cite{And74} expressed Gordon's result in terms of the following $q$-identity
\begin{align}
\sum_{n_1 \geq \dots \geq n_{k-1} \geq 0} \frac{q^{n_1^2+ \cdots +n_{k-1}^2 +n_{i}+ \cdots +n_{k-1}}}
{(q)_{n_1-n_2} \cdots (q)_{n_{k-2}-n_{k-1}} (q)_{n_{k-1}}}
=\frac{(q^{i},q^{2k-i+1},q^{2k+1};q^{2k+1})_\infty}{(q)_\infty} \label{Andrews-Gordon-identity}
\end{align}
with $k \geq 2$ and $1 \leq i \leq k$ being two integers, which is  the well-known Andrews--Gordon identity.

In 2019, Kanade and Russell \cite{KR19}  searched for Rogers--Ramanujan type identities and conjectured several new ones which are related to level
2 characters of the affine Lie algebra $A_9^{(2)}$. Since then numerous new Rogers--Ramanujan type identities have been derived, especially in the form of double sums. By using the $q$-Zeilberger algorithm to get the recurrence relation, Andrews and Uncu \cite{AU23} proved a double sum identity of Rogers--Ramanujan type as follows
\begin{align*}
\sum_{m,n\geq0} \frac{(-1)^n q^{\binom{3n+1}{2} +m^2+3mn}}{(q)_m (q^3;q^3)_n} =\frac{1}{(q;q^3)_\infty}.
\end{align*}

Note that the right hand sides of the Rogers--Ramanujan identities \eqref{RR1} and \eqref{RR2} can be seen as modular forms, which is not easy to be observed from their sum--sides.
Nahm \cite{Nahm07,Nahm941,Nahm942} considered a specific class of $q$-series as follows, which is known as Nahm sum or Nahm series  
\begin{align*}
f_{A,B,C}(q)=\sum_{{{n}}=(n_1,\ldots,n_r)^{T} \in \mathbb{N}^r} \frac{q^{\frac{1}{2} {{n}}^T A{{n}} +{{n}}^T B+C}}{(q;q)_{n_1} \cdots (q;q)_{n_r}},
\end{align*}
where $r$ is a positive integer, $A$ is a real positive definite symmetric $r\times r$ matrix, $B$ is a $r$-dimensional column vector, and $C$ is a rational number. In \cite{Nahm07}, Nahm proposed a conjecture  that characterizes the necessary and sufficient conditions on the matrix component of a modular triple $(A,B,C)$ to ensure that the associated series is a modular function.  Then Wang \cite{Wang23} generalized the identities on Nahm series by indices.  Later, Cao and Wang \cite{CA23} applying the contour integral method to establish a number of multi-sum Rogers--Ramanujan type identities including
\begin{align}
\sum_{i,j\geq0} \frac{(-1)^{i+j} q^{2a(i-j)} q^{(i^2-i+j^2-j+4a(i-j)^2)/2}}{(q)_i (q)_j} =\frac{(q, q^{4a}, q^{4a+1}; q^{4a+1})_\infty}{(q)_\infty}. \label{Wang-(3.22)}
\end{align}

In this paper, based on the orthogonality relations of  trigonometric functions and the most classical $q$-orthogonal polynomials, that is,  the (continuous) $q$-Hermite polynomials, we obtain serveral double sum Rogers--Ramanujan type identities.

The $q$-orthogonal polynomials play an important role in the study of $q$-series, which were introduced by Rogers and Szeg\"o  \cite{Rog93, Rog94, Sze26}. Based on these works, in 1985, Askey and Wilson \cite{AW85} derived one of the most general class of $q$-orthogonal polynomials, which is named Askey--Wilson polynomials. As its special case, the (continuous) $q$-Hermite polynomials is defined as follows
\begin{align}
H_n(x|q) =\sum_{k=0}^{n} {n\brack k} e^{-i(n-2k) \theta},\label{q-hermite}
\end{align}
where $x=\cos\theta$ and $H_0(x|q)=1$.

The orthogonality relation of  $q$-Hermite polynomials is
\begin{align}
\frac{(q)_\infty}{2\pi} \int_{0}^{\pi} H_n(x|q) H_m(x|q)(e^{2i\theta}, e^{-2i\theta})_\infty d\theta =\delta_{mn} (q)_n,\label{Hermite-orth}
\end{align}
where $\delta$ is the Kronecker delta function such that  $\delta_{mn}=1$, if $m=n$, otherwise $\delta_{mn}=0$.

In 1999,  Garrett, Ismail, and Stanton \cite{GIS99} applied the orthogonality relations of the $q$-Hermite and $q$-ultraspherical polynomials to study the Rogers--Ramanujan identities. In 2012, Andrews \cite{And12} demonstrated the applications of $q$-orthogonal polynomials in the study of mock theta functions. Many other $q$-orthogonal polynomials have been studied over the years by Al-Salam, Ismail, Liu, Masson and Stanton et al. in \cite{AI88,IM94,CL25}.

As a limit case of $q$-Hermite polynomials, Chebyshev polynomials of the second kind can be expressed as $U_n(x)=H_n(x|0)$. Garrett, Ismail and Stanton \cite{GIS99}  gave the following expansion formula
\begin{align}
U_n(x) =\sum_{j=0}^{\lfloor \frac{n}{2} \rfloor} {n-j \brack j} (-1)^j q^{\binom{j+1}{2}} H_{n-2j}(x|q),\label{U_n-to-Hermite}
\end{align}
which bridges the study of classical orthogonal polynomials and  $q$-orthogonal polynomials.
Recently, Andrews \cite{And12, And19} applied  orthogonal polynomials  to study various problems in $q$-series, especially to Rogers--Ramanujan type identities. Andrews obtained the following two identities in \cite{And19}
\begin{align}
&\sum_{n=0}^{\infty} \frac{q^{n^2+n} \prod_{j=1}^{n} (1+2xq^j+q^{2j})}{(q)_{2n+1}}=\frac{1}{(q)_\infty} \sum_{n=0}^{\infty} q^{3\binom{n+1}{2}} V_n(x), \label{Andrews-(4.2)}\\
&\sum_{n=0}^{\infty} \frac{q^{n^2+n} \prod_{j=0}^{n-1} (1+2xq^j+q^{2j})}{(q)_{2n}} =\frac{1}{(q)_\infty} \sum_{n=0}^{\infty} q^{n(3n+1)/2} (1-q^{2n+1}) W_n(x), \label{Andrews-(5.1)}
\end{align}
where  $V_n(x)$ and $W_n(x)$ are Chebyshev polynomials of the third and fourth kinds, respectively.

Inspired by Andrews' work, Sun \cite{Sun23} and Yao \cite{Yao25} further studied Rogers--Ramanujan type identities by constructing Bailey pairs involving Chebyshev polynomials, which leads to the following identities \cite[(3.1)]{Yao25} and \cite[(6.2)]{Sun23}:
\begin{align}
&\sum_{n=0}^{\infty} \frac{q^{n^2 +2n} \prod_{j=1}^{n} (1 +2x q^j +q^{2j})}{(1 +q^{n+1}) (q)_{2n+1}} =\frac{1}{(q)_{\infty}} \sum_{n=0}^{\infty} (1 -q^{n+1}) q^{n(3n+5)/2} U_n(x),
\label{Yao-(3.1)}\\[2mm]
&\sum_{n=0}^\infty \frac{q^{n^2+n} (-1;q^2)_n \prod_{i=1}^n (1+2xq^{2i-1}+q^{4i-2})}{(q^2;q^2)_{2n}} \label{Sun-(6.2)}\\
&\qquad \qquad =\frac{(-q^2;q^2)_\infty}{(q^2;q^2)_{\infty}} \sum_{n=0}^\infty \frac{(-1;q^2)_n q^{2n^2+n}} {(-q^2;q^2)_n} \big( V_n(x)+V_{n-1}(x) \big).\nonumber
\end{align}

From these identities on Chebyshev polynomials, by expanding $U_n(x)$ using formula \eqref{U_n-to-Hermite} and then applying the orthogonality relation \eqref{Hermite-orth} of $q$-Hermite polynomials, we obtain some identities of Rogers--Ramanujan type on double sums. For example, from \eqref{Yao-(3.1)} we get
\begin{align}
\sum_{n=0}^{\infty} \sum_{m=-\lfloor \frac{n+1}{2} \rfloor}^{\lfloor \frac{n}{2} \rfloor} \frac{(-1)^m q^{n(n+2)+m(5m+3)/2}} {(1+q^{n+1}) (q)_{2n+1}} {2n+1 \brack n-2m}=\frac{(q, q^{12}, q^{13};q^{13})_\infty}{(q)_\infty}.\label{Yao-(3.1)-orth-simply}
\end{align}
Note that identity \eqref{Yao-(3.1)-orth-simply} can be seen as a new form of the Andrews--Gordon identity \eqref{Andrews-Gordon-identity} with $k=6, i=1$, in which it takes the form of a quintuple sum. It also coincides with \eqref{Wang-(3.22)} with  $a=3$.

This paper is organized as follows. In Section~2, we introduce some basic definitions and notations for $q$-series and orthogonal polynomials. In Section~3, we obtain a series of integral identities on $q$-Hermite polynomials based on the $q$-binomial theorem. In Section~4, from identities on Chebyshev polynomials, by using the integral identities obtained in Section~3, we derive Rogers--Ramanujan type identities on double sums such as \eqref{Yao-(3.1)-orth-simply} and also an identity on Appell--Lerch series. In Section~5, by transforming Chebyshev polynomials $V_n(x)$ and $W_n(x)$ into $U_n(x)$, we also obtain triple-sum identities with a free variable $\ell$. When $\ell=0$, it leads to more Rogers--Ramanujan type identities on double sums.

\section{Preliminaries}

Throughout this paper, we adopt  standard notations and terminologies
for $q$-series \cite{GR04} and we assume that $|q|<1$. The $q$-shifted factorials are given by
\begin{align}\label{q-standard notations}
(a)_n=(a;q)_n=\begin{cases}
1, & \text{\it if $n=0$}, \\[2mm]
\prod\limits_{j=1}^{n} (1-aq^{j-1}), & \text{\it if $n\geq 1$}, 
\end{cases}
\end{align}
and
\[
(a)_\infty =(a;q)_\infty =\prod_{n=0}^\infty (1-aq^n).
\]
There are more compact notations for the multiple $q$-shifted factorials:
\begin{align*}
&(a_1, a_2, \ldots, a_m)_n=(a_1, a_2, \ldots, a_m;q)_n=(a_1;q)_n (a_2;q)_n \cdots (a_m;q)_n,\\
&(a_1, a_2, \ldots, a_m)_{\infty}=(a_1, a_2, \ldots, a_m;q)_{\infty}=(a_1;q)_{\infty} (a_2;q)_{\infty} \cdots (a_m;q)_{\infty}.
\end{align*}
The $q$-binomial coefficients, or Gaussian polynomials are given by
\[
{n \brack k}_q=
  \dfrac{(q;q)_n} {(q;q)_k(q;q)_{n-k}},
\]
where the subscript ``$q$" is usually omitted if no confusion arises.
The  $q$-binomial theorem  is known as \cite[(3.3.6)]{And84}
\begin{align}
\sum^{n}_{j=0} {n\brack j} (-1)^j q^{\binom{j}{2}} z^j=(z)_n. \label{q-binomial-thm}
\end{align}

The well--known Jacobi's triple product identity \cite[(II.28)]{GR04} is
\begin{align}
\sum_{n=-\infty}^{\infty} (-1)^n q^{\binom{n}{2}} z^n =(z, q/z, q)_\infty. \label{JTP}
\end{align}

Recall that Chebyshev polynomials $T_n(x), U_n(x), V_n(x), W_n(x)$ of the first, second, third and fourth kinds are defined  as follows
\begin{align*}
&T_n(x)=\cos(n \theta),\qquad\qquad
U_n(x)=\frac{\sin(n+1)\theta}{\sin\theta},\\
&V_n(x)=\frac{\cos(n +\frac{1}{2}) \theta}{\cos \frac{1}{2} \theta}, \qquad
W_n(x)=\frac{\sin(n +\frac{1}{2}) \theta}{\sin \frac{1}{2} \theta},
\end{align*}
where $x=\cos \theta$ and the initial conditions are $T_0(x)=U_0(x)=V_0(x)=W_0(x)=1$. For convenience, we also set $U_n(x)=V_n(x)=W_n(x)=0$ for $n< 0$. In fact, these four polynomials satisfy the same recurrence relation
\begin{align*}
p_n(x)=2x p_{n-1}(x) -p_{n-2}(x), \qquad (n\geq 2)
\end{align*}
with $p_0(x)=1$, and $p_1(x)=x, 2x, 2x-1, 2x+1$, respectively. They are also closely related to each other. As given in \cite{MH03} for $n\geq 1$, we see that
\begin{align}
V_n(x)=U_n(x)-U_{n-1}(x),\label{V_n}
\end{align}
and
\begin{align}
W_n(x)=U_n(x)+U_{n-1}(x).\label{W_n}
\end{align}

\section{Integral identities on $q$-Hermite polynomials}

In this section, based on the $q$-binomial theorem, we obtain a series of integral identities on $q$-Hermite polynomials, which will be used to derive Rogers--Ramanujan type identities.

By following  Rogers' way \cite{Rog17} to treat $q$-polynomials in terms of Fourier series, we replace $x$ by $\cos \theta=(e^{i\theta}+e^{-i\theta})/2$. As pointed out by Andrews \cite{And19}, the following formula is a finite version of the result due to Rogers \cite{Rog17}
\begin{align}
(-qe^{i\theta}, -qe^{-i\theta})_n =\prod_{j=1}^n (1+2q^j \cos\theta + q^{2j}). \label{triangle fomular}
\end{align}
From the orthogonality relations of  trigonometric functions, it is obvious that
\begin{align}\label{integral-value}
\frac{1}{2\pi} \int_{-\pi}^\pi e^{i m\theta} d \theta=\begin{cases}
1, & \text{\it if $m=0$}, \\[2mm]
0, & \text{\it otherwise}.
\end{cases}
\end{align}
By combining with the orthogonality relation of $q$-Hermite polynomials,  we derive the following integral identity.

\begin{lemma}\label{integral1-generalize-2hermite}
We have
\begin{align} \label{int-eq-2hermite}
&\frac{(q)_\infty}{2\pi} \int_{-\pi}^{\pi} H_{\ell_1} (x|q) H_{\ell_2} (x|q) (-qe^{i\theta}, -qe^{-i\theta})_n (e^{2i\theta}, e^{-2i\theta})_\infty \ d\theta \\
&= \sum_{{\begin{subarray}{c} k_i=0 \\ i=1,2
\end{subarray}}}^{\ell_i} \sum_{m=-\infty}^{\infty} (-1)^m q^{\binom{2m+2k_1+2k_2-\ell_1-\ell_2+1}{2}+ \binom{m+1}{2}} \frac{(1+q^{n+1})(1-q^{2m+2k_1+2k_2-\ell_1-\ell_2+1})} {1-q^{n+2m+2k_1+2k_2-\ell_1-\ell_2+2}} \nonumber  \\
&\qquad\qquad\qquad
\cdot {\ell_1 \brack k_1}{\ell_2 \brack k_2}  {2n+1 \brack  n-2m-2k_1-2k_2+\ell_1+\ell_2}, \nonumber
\end{align}
where $m$ ranges over all the integers such that the $q$-binomial coefficients in the above sum is not equal to zero.
\end{lemma}

\begin{proof} Applying the $q$-binomial theorem \eqref{q-binomial-thm}, we have that
\begin{align*}
(-qe^{i\theta}, -qe^{-i\theta})_n & =(-qe^{i\theta})_n\ q^{\binom{n+1}{2}} e^{-in\theta}  (-q^{-n}e^{i\theta})_n \\
&=\frac{q^{\binom{n+1}{2}}}{(1+e^{i\theta}) e^{in\theta}} (-q^{-n}e^{i\theta})_{2n+1} \\
&=\frac{1}{1+e^{i\theta}} \sum_{j=0}^{2n+1} {2n+1 \brack j} q^{\binom{n+1}{2}+ \binom{j}{2}-nj} e^{i(j-n)\theta}. 
\end{align*}
Substituting it into the left hand side of \eqref{int-eq-2hermite}, using the definition of $q$-Hermite polynomials \eqref{q-hermite} and then applying Jacobi's triple product identity \eqref{JTP}, it becomes
 \begin{align*}
&\frac{(q)_\infty}{2\pi}\sum_{{\begin{subarray}{c} k_i=0 \\ i=1,2
\end{subarray}}}^{\ell_i}  \sum_{j=0}^{2n+1} q^{\binom{n+1}{2} +\binom{j}{2}-nj} {2n+1 \brack j} {\ell_1 \brack k_1}{\ell_2 \brack k_2} \\
&\ \cdot\int_{-\pi}^{\pi} (1-e^{i\theta}) \ e^{(2k_1+2k_2-\ell_1-\ell_2+j-n) i\theta} (qe^{2i\theta}, e^{-2i\theta})_\infty \ d\theta \\ &=\sum_{{\begin{subarray}{c} k_i=0 \\ i=1,2
\end{subarray}}}^{\ell_i}  \sum_{j=0}^{2n+1} \sum_{m=-\infty}^{\infty} (-1)^m q^{\binom{n+1}{2}+ \binom{m+1}{2} +\binom{j}{2}-nj} {\ell_1 \brack k_1}{\ell_2 \brack k_2} {2n+1 \brack j} \\
&\qquad \cdot \frac{1}{2\pi} \int_{-\pi}^{\pi} (1-e^{i\theta}) e^{(2m+2k_1+2k_2-\ell_1-\ell_2+j-n)i\theta}\ d\theta.
\end{align*}
Using   \eqref{integral-value} to compute the integral and simplifying, it completes the proof.
\end{proof}

Now by taking $\ell_1=0$ or $\ell_2=0$, and noting $H_0 (x|q)=1$, the above lemma reduces to the following result.

\begin{corollary}\label{integral1-generalize}
We have
\begin{align*}
&\frac{(q)_\infty}{2\pi} \int_{-\pi}^{\pi} \  H_\ell (x|q) \ (-qe^{i\theta}, -qe^{-i\theta})_n \ (e^{2i\theta}, e^{-2i\theta})_\infty\ d\theta\\
&\quad =\sum_{m=-\infty}^{\infty} \sum_{k=0}^{\ell}  (-1)^m q^{\binom{2m+2k-\ell+1}{2}+ \binom{m+1}{2}}   {\ell \brack k}\\
&\qquad
\cdot \left({2n+1 \brack n-2m-2k+\ell} -q^{2m+2k-\ell+1} {2n+1 \brack n-2m-2k+\ell-1} \right).
\end{align*}
\end{corollary}

\begin{proof}
By taking $\ell_2=0$ and substituting $\ell_1$ by $\ell$ in Lemma \ref{integral1-generalize-2hermite}, it turns to be
\begin{align*}
&\sum_{k=0}^{\ell} \sum_{m=-\infty}^{\infty} (-1)^m q^{\binom{2m+2k-\ell+1}{2}+ \binom{m+1}{2}} \frac{(1+q^{n+1})(1-q^{2m+2k-\ell+1})} {1-q^{n+2m+2k-\ell+2}}\\
&\qquad\qquad\qquad\qquad\qquad\qquad\qquad\qquad\qquad
\cdot {\ell \brack k} {2n+1 \brack n-2m-2k+\ell}\\[2mm]
&=\sum_{m=-\infty}^{\infty} \sum_{k=0}^{\ell}  (-1)^m q^{\binom{2m+2k-\ell+1}{2}+ \binom{m+1}{2}} {\ell \brack k} {2n+1 \brack n-2m-2k+\ell} \\
&\qquad
-\sum_{m=-\infty}^{\infty} \sum_{k=0}^{\ell}  (-1)^m q^{\binom{2m+2k-\ell+2}{2}+ \binom{m+1}{2}} {\ell \brack k} {2n+1 \brack n-2m-2k+\ell-1},
\end{align*}
which completes the proof by simplifying.
\end{proof}

Then by taking $\ell=0$ in Corollary \ref{integral1-generalize} and simplifying, we get the result as follows.

\begin{corollary}\label{integral1}
We have
\begin{align}
&\frac{(q)_\infty}{2\pi} \int_{0}^{\pi} (-qe^{i\theta}, -qe^{-i\theta})_n \ (e^{2i\theta}, e^{-2i\theta})_\infty \ d\theta \\
&\qquad\qquad\qquad\qquad\qquad
=\sum_{m=-\lfloor \frac{n+1}{2} \rfloor}^{\lfloor \frac{n}{2} \rfloor} (-1)^m q^{m(5m+3)/2} {2n+1 \brack n-2m}.\nonumber
\end{align}
\end{corollary}

\begin{proof}
Taking $\ell=0$ in Corollary \ref{integral1-generalize}, and by noting $H_0 (x|q)=1$, it reduces to
\begin{align*}
&\frac{(q)_\infty}{2\pi} \int_{-\pi}^{\pi} (-qe^{i\theta}, -qe^{-i\theta})_n (e^{2i\theta}, e^{-2i\theta})_\infty \ d\theta \\
&\quad =\sum_{m=-\infty}^{\infty} (-1)^m q^{\binom{2m+1}{2}+\binom{m+1}{2}} \left({2n+1 \brack n-2m}- q^{2m+1} {2n+1 \brack n-2m-1} \right).
\end{align*}
Letting $m\mapsto -m-1$ in the second term on the right hand side of the above identity, and noting the  integrand on the left hand side is an even function, the result is proved.
\end{proof}

Note that by further taking $n \rightarrow \infty$ on the right hand side of Corollary \ref{integral1} , it leads to the right hand side of the second Rogers--Ramanujan identity \eqref{RR2}.

\begin{lemma}\label{integral2-generalize}
We have
\begin{align}
&\frac{(q)_\infty}{2\pi} \int_{-\pi}^{\pi} H_\ell (x|q) \ (-qe^{i\theta}, -qe^{-i\theta}; q^2)_n \ (e^{2i\theta}, e^{-2i\theta})_\infty \ d\theta \label{integral2-generalize-sub} \\
&\quad
=\sum_{m=-\infty}^{\infty} \sum_{k=0}^{\ell}  (-1)^m q^{(2m+2k-\ell)^2+ \binom{m+1}{2}} {\ell \brack k}  \nonumber \\
&\qquad
\cdot \left({2n \brack n-2m-2k+\ell}_{q^2}
-q^{4(2m+2k-\ell+1)} {2n \brack n-2m-2k+\ell-2}_{q^2} \right).\nonumber
\end{align}
\end{lemma}

\begin{proof}
It is easy to see that
\begin{align*}
(z,q/z)_n =(-1)^n q^{n+1\choose 2} z^{-n} (zq^{-n};q)_{2n}.
\end{align*}
By using the $q$-binomial theorem \eqref{q-binomial-thm}, we obtain that
\begin{align}
(z,q/z)_n =\sum_{j=0}^{2n} (-1)^{j+n} z^{j-n} q^{\binom{n+1}{2}+\binom{j}{2}-nj} {2n \brack j}.\label{basic2}
\end{align}
Taking $q\mapsto q^2, z\mapsto -qe^{i\theta}$ in \eqref{basic2}, and substituting it into the left hand side of \eqref{integral2-generalize-sub}, then following the similar procedures as given in the proof of Corollary \ref{integral1-generalize}, we obtain the result.
\end{proof}

Now by taking $\ell=0$ in Corollary \ref{integral2-generalize}, it leads to the following identity.

\begin{corollary}\label{integral2}
We have
\begin{align}
&\frac{(q)_\infty}{2\pi} \int_{0}^{\pi} (-qe^{i\theta}, -qe^{-i\theta}; q^2)_n \ (e^{2i\theta}, e^{-2i\theta})_\infty \ d\theta \\
&\qquad\qquad\qquad\qquad\qquad
=\sum_{m=-\lfloor \frac{n}{2} \rfloor}^{\lfloor \frac{n}{2} \rfloor} (-1)^m q^{m(9m+1)/2} {2n \brack n-2m}_{q^2}.\nonumber
\end{align}
\end{corollary}

\begin{lemma}\label{integral3-generalize-a}
We have
\begin{align}
&\frac{(q)_\infty}{2\pi} \int_{-\pi}^{\pi} H_\ell(x|q)\ (-e^{i\theta}, -e^{-i\theta};q^a)_n \ (e^{2i\theta}, e^{-2i\theta})_\infty \ d\theta \\
&\ =\sum_{m=-\infty}^{\infty}\sum_{k=0}^{\ell}  (-1)^m q^{a\binom{2m+2k-\ell}{2}+ \binom{m+1}{2}} {\ell \brack k} \nonumber\\
&\quad
\cdot \Bigg({2n-1 \brack n-2m-2k+\ell}_{q^a}
+q^{a(2m+2k-\ell)}{2n-1 \brack n-2m-2k+\ell-1}_{q^a} -q^{2a(2m+2k-\ell)+a} \nonumber\\
&\qquad\qquad
\cdot {2n-1 \brack n-2m-2k+\ell-2}_{q^a} -q^{3a(2m+2k-\ell+1)} {2n-1 \brack n-2m-2k+\ell-3}_{q^a} \Bigg). \nonumber
\end{align}
\end{lemma}

\begin{proof}
By using the $q$-binomial theorem \eqref{q-binomial-thm}, we have that
\begin{align}
(-z, -1/z;q^a)_n =(1+z) \sum_{j=0}^{2n-1} q^{a \big( \binom{j}{2} +\binom{n}{2}-j(n-1) \big)} z^{j-n} {2n-1 \brack j}_{q^a}. \label{basic3}
\end{align}
Taking $z=e^{i\theta}$ in \eqref{basic3} and following the similar procedures as given in the proof of  Corollary \ref{integral1-generalize}, the proof is complete.
\end{proof}

When we take $a=1$ in Lemma \ref{integral3-generalize-a}, it reduces to the following result.

\begin{corollary}\label{integral3-generalize}
We have
\begin{align}
&\frac{(q)_\infty}{2\pi} \int_{-\pi}^{\pi} H_\ell(x|q)\ (-e^{i\theta}, -e^{-i\theta})_n \ (e^{2i\theta}, e^{-2i\theta})_\infty \ d\theta \\
&\ =\sum_{m=-\infty}^{\infty}\sum_{k=0}^{\ell}  (-1)^m q^{\binom{2m+2k-\ell}{2}+ \binom{m+1}{2}} {\ell \brack k} \nonumber\\
&\quad
\cdot \Bigg({2n-1 \brack n-2m-2k+\ell}
+q^{2m+2k-\ell}{2n-1 \brack n-2m-2k+\ell-1}
-q^{2(2m+2k-\ell)+1} \nonumber\\
&\qquad\qquad
\cdot {2n-1 \brack n-2m-2k+\ell-2} -q^{3(2m+2k-\ell+1)} {2n-1 \brack n-2m-2k+\ell-3} \Bigg). \nonumber
\end{align}
\end{corollary}

By taking $a=2$ in Lemma \ref{integral3-generalize-a}, it gives the following identity.

\begin{corollary}\label{integral4-generalize}
We have
\begin{align}
&\frac{(q)_\infty}{2\pi} \int_{-\pi}^{\pi} H_\ell(x|q)\ (-e^{i\theta}, -e^{-i\theta}; q^2)_n \ (e^{2i\theta}, e^{-2i\theta})_\infty \ d\theta \\
&\ =\sum_{m=-\infty}^{\infty}\sum_{k=0}^{\ell}  (-1)^m q^{2 \binom{2m+2k-\ell}{2}+ \binom{m+1}{2}} {\ell \brack k} \nonumber\\
&\quad
\cdot \Bigg({2n-1 \brack n-2m-2k+\ell}_{q^2} +q^{2(2m+2k-\ell)} {2n-1 \brack n-2m-2k+\ell-1}_{q^2}
-q^{4(2m+2k-\ell)+2}  \nonumber\\
&\qquad\qquad
\cdot {2n-1 \brack n-2m-2k+\ell-2}_{q^2} -q^{6(2m+2k-\ell+1)} {2n-1 \brack n-2m-2k+\ell-3}_{q^2} \Bigg).\nonumber
\end{align}
\end{corollary}

Setting $\ell=0$ in Corollary \ref{integral4-generalize}, we are led to the following identity.

\begin{corollary}\label{integral4}
We have
\begin{align}
&\frac{(q)_\infty}{2\pi} \int_{0}^{\pi} (-e^{i\theta}, -e^{-i\theta}; q^2)_n \ (e^{2i\theta}, e^{-2i\theta})_\infty \ d\theta \\
&\qquad\qquad\qquad
=\sum_{m=-\lfloor \frac{n-1}{2} \rfloor}^{\lfloor \frac{n}{2} \rfloor} (-1)^m q^{m(9m-5)/2} (1+q^m) {2n-1 \brack n-2m}_{q^2}.\nonumber
\end{align}
\end{corollary}

\section{Double sum Rogers--Ramanujan type identities}

By using the theory of Bailey's transform,  Sun \cite{Sun23} and Yao \cite{Yao25}  obtained several identities involving Chebyshev polynomials. In this section, based on these identities, combining the expansion formula \eqref{U_n-to-Hermite} of  Chebyshev polynomials  with the orthogonality relation of  $q$-Hermite polynomials \eqref{Hermite-orth}, we derive some Rogers--Ramanujan type identities on double sums and also an identity on Appell--Lerch series. Recall that the Appell--Lerch series has the following form
\begin{align*}
\sum_{n=-\infty}^\infty \frac{(-1)^{\ell n} q^{\ell n(n+1)/2} b^n}{1-aq^n},
\end{align*}
which plays an important role in the theory of $q$-series, partitions, and modular forms.

First, we illustrate the detailed procedures by showing the double sum identity \eqref{Yao-(3.1)-orth-simply} of the Rogers--Ramanujan type.

{\noindent \bf Proof of identity  \eqref{Yao-(3.1)-orth-simply}.}
The first step is that from identity \eqref{Yao-(3.1)} due to  Yao \cite[(3.1)]{Yao25},  expanding $U_n(x)$ by using formula \eqref{U_n-to-Hermite} and noting \eqref{triangle fomular}, it turns to be
\begin{align*}
&\sum_{n=0}^{\infty} \frac{q^{n(n+2)} (-qe^{i \theta}, -qe^{-i \theta})_n}{(1 +q^{n+1}) (q)_{2n+1}} \\
&\qquad
=\frac{1}{(q)_{\infty}} \sum_{n=0}^{\infty} (1 -q^{n+1}) q^{n(3n+5)/2} \sum_{j=0}^{\lfloor \frac{n}{2} \rfloor} {n-j \brack j} (-1)^j q^{\binom{j+1}{2}} H_{n-2j} (x|q).
\end{align*}
Next, multiplying both sides by
\begin{align*}
\frac{(q)_\infty}{2\pi} (e^{2i\theta}, e^{-2i\theta})_\infty, 
\end{align*}
and then calculating the integrals from $0$ to $\pi$, we are led to
\begin{align*}
&\sum_{n=0}^{\infty} \frac{q^{n(n+2)}}{(1+q^{n+1}) (q)_{2n+1}} \cdot \frac{(q)_\infty}{2\pi} \int_{0}^{\pi} (-qe^{i \theta}, -qe^{-i \theta})_n (e^{2i\theta}, e^{-2i\theta})_\infty \ d\theta\\
&\ =\frac{1}{(q)_{\infty}} \sum_{n=0}^{\infty} (1 -q^{n+1}) q^{n(3n+5)/2} \sum_{j=0}^{\lfloor \frac{n}{2} \rfloor} {n-j \brack j} (-1)^j q^{\binom{j+1}{2}} \\
&\qquad\qquad\qquad \qquad
\cdot \frac{(q)_\infty}{2\pi} \int_{0}^{\pi} H_{n-2j}(x|q) (e^{2i\theta}, e^{-2i\theta})_\infty \ d\theta.
\end{align*}
Then, for the left hand side, it can be simplified by Corollary \ref{integral1} directly. For the right hand side,  by using the orthogonality relation of  $q$-Hermite polynomials \eqref{Hermite-orth}, it implies that
\begin{align*}
\frac{1}{(q)_\infty} \sum_{n=0}^{\infty} (-1)^j q^{\binom{j+1}{2}+j(6j+5)} (1-q^{2j+1}).
\end{align*}
Finally, by applying  Jacobi's triple product identity \eqref{JTP}, the proof is complete. $\Box$

By using the similar procedures given as above, from the  identities  due to \cite[(3.14),(3.22)]{Yao25}
{\small \begin{align*}
&\sum_{n=0}^{\infty} \frac{q^{\frac{n(n+1)}{2}} (-q)_n \prod_{j=1}^{n} (1 +2xq^j +q^{2j})}{(q)_{2n+1}}
=\frac{(q^2; q^2)_{\infty}}{(q)_{\infty}^2} \sum_{n=0}^{\infty} q^{n^2 + n} \big( U_n(x) - U_{n-1}(x) \big), \\ 
&\sum_{n=0}^{\infty} \frac{q^{n^2} \prod_{j=1}^{n} (1 +2xq^j +q^{2j})}{(q)_{2n+1}}
=\frac{1}{(q)_{\infty}} \sum_{n=0}^{\infty} q^{n(3n+1)/2} (1 + q^{2n+1}) \big(U_n(x) - U_{n-1}(x) \big),   
\end{align*}}
we obtain the following identities of Rogers--Ramanujan type.

\begin{theorem}
We have
\begin{subequations}
\begin{align}
&\sum_{n=0}^{\infty} \sum_{m=-\lfloor \frac{n+1}{2} \rfloor}^{\lfloor \frac{n}{2} \rfloor} \frac{(-1)^m q^{n(n+1)/2 +m(5m+3)/2}} {(q)_n (q;q^2)_{n+1}} {2n+1 \brack n-2m} \label{Yao-(3.14)-orth-simply}\\
&\qquad\qquad\qquad\qquad\qquad\qquad\qquad\quad
=\frac{(-q)_\infty (q^2, q^7, q^9; q^9)_\infty}{(q)_\infty},  \nonumber\\
&\sum_{n=0}^{\infty} \sum_{m=-\lfloor \frac{n+1}{2} \rfloor}^{\lfloor \frac{n}{2} \rfloor} \frac{(-1)^m q^{n^2 +m(5m+3)/2}} {(q)_{2n+1}} {2n+1 \brack n-2m}
\label{Yao-(3.22)-orth-simply} \\
&\qquad\qquad\qquad
=\frac{1}{(q)_\infty} \big( (q^5, q^{8}, q^{13}; q^{13})_\infty +q(q, q^{12}, q^{13}; q^{13})_\infty \big) . \nonumber
\end{align}
\end{subequations}
\end{theorem}
The identity \eqref{Yao-(3.14)-orth-simply} yields a new form of the Andrews--Gordon type identity with $k=5, i=2$, which is also studied by Sang and Shi in \cite[(1.4)]{SS15}.

Based on the identities \cite[(5.1)]{Sun23} and \cite[(3.29),(3.42)]{Yao25}, respectively,
\begin{align*}
&\sum_{n=0}^\infty \frac{q^{n^2} (-q; q^2)_n \prod_{i=1}^n (1 +2xq^{2i-1} +q^{4i-2})}{(q^2; q^2)_{2n}} \\
&\qquad\qquad\qquad\qquad\qquad\qquad
=\frac{(-q; q^2)_{\infty}}{(q^2; q^2)_{\infty}} \sum_{n=0}^\infty q^{2n^2} \big( V_n(x) +V_{n-1}(x) \big), \\ 
&\sum_{n=0}^{\infty} \frac{q^{2n^2+2n} \prod_{j=1}^{n} (1 +2xq^{2j-1} +q^{4j-2})} {(1+q^{2n+1})(q^2; q^2)_{2n}} \nonumber\\
&\qquad\qquad\qquad
=\frac{1}{(q^2; q^2)_{\infty}} \sum_{n=0}^{\infty} (1 -q^{2n+1}) q^{3n^2+2n} \big( U_n(x) +U_{n-1}(x) \big), \\ 
&\sum_{n=0}^{\infty} \frac{q^{n^2+n} (-q^2;q^2)_n \prod_{j=1}^{n} (1 +2xq^{2j-1} +q^{4j-2})}{(1 +q^{2n+1})(q^2;q^2)_{2n}} \\
&\qquad\qquad\qquad
=\frac{(q^4; q^4)_{\infty}}{(q^2; q^2)_{\infty}^2} \sum_{n=0}^{\infty} (1-q^{2n+1}) q^{2n^2+n} \big( U_n(x) +U_{n-1}(x) \big), 
\end{align*}
and by using Corollary \ref{integral2}, we derive the following identities.

\begin{theorem}
We have
\begin{align*}
&\sum_{n=0}^{\infty} \sum_{m=-\lfloor \frac{n}{2} \rfloor}^{\lfloor \frac{n}{2} \rfloor} \frac{(-1)^m q^{n^2+m(9m+1)/2}} {(q; q^2)_n  (q^4; q^4)_n } {2n \brack n-2m}_{q^2} =\frac{(-q; q^2)_\infty (q^8, q^{9}, q^{17}; q^{17})_\infty}{(q^2; q^2)_\infty}, \\
&\sum_{n=0}^{\infty} \sum_{m=-\lfloor \frac{n}{2} \rfloor}^{\lfloor \frac{n}{2} \rfloor} \frac{(-1)^m q^{2n(n+1) +m(9m+1)/2}} {(1+q^{2n+1}) (q^2;q^2)_{2n}} {2n \brack n-2m}_{q^2}\\
&\qquad\qquad\qquad\qquad\quad
=\frac{1}{(q^2;q^2)_\infty} \big( (q^8, q^{17}, q^{25}; q^{25})_\infty -q(q^4, q^{21}, q^{25}; q^{25})_\infty \big),\\
&\sum_{n=0}^{\infty} \sum_{m=-\lfloor \frac{n}{2} \rfloor}^{\lfloor \frac{n}{2} \rfloor} \frac{(-1)^m q^{n(n+1) +m(9m+1)/2}} {(q)_{2n} (-q;q^2)_{n+1}} {2n \brack n-2m}_{q^2} \\
&\qquad\qquad\qquad\qquad
=\frac{(-q^2;q^2)_\infty}{(q^2;q^2)_\infty} \big( (q^6, q^{11}, q^{17}; q^{17})_\infty -q(q^2, q^{15}, q^{17}; q^{17})_\infty \big).
\end{align*}
\end{theorem}

By noting the identity given by Yao in \cite[(3.60)]{Yao25}
\begin{align*}
&1 +\sum_{n=1}^{\infty} \frac{q^{n^2} \prod_{j=0}^{n-1} (1 +2xq^j +q^{2j})}{(q;q)_{2n}}\\
&\qquad\qquad\qquad\qquad
=\frac{1}{(q)_{\infty}} \sum_{n=0}^{\infty} q^{n(3n-1)/2} (1-q^{4n+2}) \big( U_n(x) + U_{n-1}(x) \big), 
\end{align*}
and applying Corollary \ref{integral4}, we obtain the following identity.

\begin{theorem}
We have
\begin{align}\label{Yao-(3.60)-orth-simply}
&1+\sum_{n=1}^{\infty} \sum_{m=-\lfloor \frac{n-1}{2} \rfloor}^{\lfloor \frac{n}{2} \rfloor} \frac{(-1)^m q^{n^2 +m(5m-3)/2} (1+q^m)} {(q)_{2n}} {2n-1 \brack n-2m} \\
&\qquad\qquad\qquad\qquad
=\frac{1}{(q)_\infty} \big( (q^6, q^{7}, q^{13}; q^{13})_\infty +q(q, q^{12}, q^{13}; q^{13})_\infty \big).\nonumber
\end{align}
\end{theorem}

Based on \eqref{Sun-(6.2)} which is given by Sun in \cite[(6.2)]{Sun23}, we also derive an identity on Appell--Lerch series.
\begin{theorem}
We have
\begin{align*}
&\sum_{n=0}^{\infty} \sum_{m=-\lfloor \frac{n}{2} \rfloor}^{\lfloor \frac{n}{2} \rfloor} \frac{(-1)^m q^{n(n+1)+m(9m+1)/2} (-1;q^2)_n} {(q^2;q^2)_{2n}} {2n \brack n-2m}_{q^2} \\
&\qquad\qquad\qquad\qquad\qquad\qquad
=\frac{2(-q^2; q^2)_\infty}{(q^2; q^2)_\infty} \sum_{n=-\infty}^{\infty} \frac{(-1)^n q^{n(17n+5)/2}}{1+q^{4n}}.
\end{align*}
\end{theorem}

\begin{proof}
By substituting \eqref{V_n} into  identity \eqref{Sun-(6.2)}, it yields that
\begin{align*}
&\sum_{n=0}^\infty \frac{q^{n(n+1)} (-1;q^2)_n (-qe^{i\theta}, -qe^{-i\theta};q^2)_n} {(q^2;q^2)_{2n}}\\
&\qquad \qquad =\frac{2(-q^2;q^2)_\infty}{(q^2;q^2)_{\infty}} \sum_{n=0}^\infty q^{n(2n+1)} \Big( \frac{1}{1+q^{2n}} -\frac{q^{8n+10}}{1+q^{2n+4}} \Big) U_n(x).
\end{align*}
Then following the similar procedures as given in the proof of identity \eqref{Yao-(3.1)-orth-simply}, and applying Corollary \ref{integral2}, we complete the proof.
\end{proof}

\section{Triple-sum identities with $\ell$}

In this section, by transforming  the terms on $V_n(x)$ and $W_n(x)$ into $U_n(x)$ and applying the orthogonality relation of $q$-Hermite polynomials \eqref{Hermite-orth}, we obtain triple-sum identities with a free variable $\ell$. When $\ell=0$, it will reduce to double-sum Rogers--Ramanujan type identities.

\begin{theorem}\label{Andrews-(4.2)-orth}
We have
\begin{align}
&\sum_{n=0}^{\infty} \frac{q^{n(n+1)}}{(q)_{2n+1}} \sum_{m=-\infty}^{\infty} \sum_{k=0}^{\ell} (-1)^m q^{\binom{2m+2k-\ell+1}{2}+ \binom{m+1}{2}} {\ell \brack k} \\
&\quad
\cdot \left({2n+1 \brack n-2m-2k+\ell}
-q^{2m+2k-\ell+1} {2n+1 \brack n-2m-2k+\ell-1} \right) \nonumber \\
&=\frac{2}{(q)_\infty} \sum_{j=0}^{\infty} (-1)^j q^{3\binom{2j+\ell+1}{2} + \binom{j+1}{2}} (1-q^{6j+3\ell+3}) (q^{j+1})_{\ell}. \nonumber
\end{align}
\end{theorem}

\begin{proof}
Rewriting Andrews' identity \eqref{Andrews-(4.2)} on $V_n(x)$ by using  \eqref{V_n} and \eqref{triangle fomular} into  $U_n(x)$, then expanding it by  \eqref{U_n-to-Hermite}, we have that
\begin{align}
&\sum_{n=0}^{\infty} \frac{q^{n^2+n} (-qe^{i\theta}, -qe^{-i\theta})_n}{(q)_{2n+1}} \label{Andrews-(4.2)-transform} \\
& \quad
=\frac{1}{(q)_\infty} \sum_{n=0}^{\infty} q^{3\binom{n+1}{2}} (1-q^{3n+3})\ U_n(x) \nonumber\\
& \quad
=\frac{1}{(q)_\infty} \sum_{n=0}^{\infty} q^{3\binom{n+1}{2}} (1-q^{3n+3}) \sum_{j=0}^{\lfloor \frac{n}{2} \rfloor} {n-j \brack j} (-1)^j q^{\binom{j+1}{2}} H_{n-2j}(x|q).\nonumber
\end{align}
Multiplying both sides by
\begin{align}
\frac{(q)_\infty}{2\pi} H_{\ell}(x|q) (e^{2i\theta}, e^{-2i\theta})_\infty, \label{multiply-term}
\end{align}
and integrating the identity from $-\pi$ to $\pi$, then the right side of \eqref{Andrews-(4.2)-transform} becomes
\begin{align*}
&\frac{1}{(q)_\infty} \sum_{n=0}^{\infty} q^{3\binom{n+1}{2}} (1-q^{3n+3}) \sum_{j=0}^{\lfloor \frac{n}{2} \rfloor} {n-j \brack j} (-1)^j q^{\binom{j+1}{2}} \\[2mm]
&\quad \cdot \frac{(q)_\infty}{2\pi} \int_{-\pi}^{\pi} H_{n-2j}(x|q) H_\ell(x|q) \ (e^{2i\theta}, e^{-2i\theta})_\infty \ d\theta.
\end{align*}
Note that the above integrand is an even function, it becomes
\begin{align*}
&\frac{2}{(q)_\infty} \sum_{n=0}^{\infty} q^{3\binom{n+1}{2}} (1-q^{3n+3}) \sum_{j=0}^{\lfloor \frac{n}{2} \rfloor} {n-j \brack j} (-1)^j q^{\binom{j+1}{2}} \\[2mm]
&\quad \cdot \frac{(q)_\infty}{2\pi} \int_{0}^{\pi} H_{n-2j}(x|q) H_\ell(x|q) \ (e^{2i\theta}, e^{-2i\theta})_\infty \ d\theta.
\end{align*}
By the orthogonality relation \eqref{Hermite-orth}, it can be further simplified to be
\begin{align*}
&\frac{2}{(q)_\infty} \sum_{j=0}^{\infty} (-1)^j q^{3\binom{2j+\ell+1}{2} + \binom{j+1}{2}} (1-q^{6j+3\ell+3}) \frac{(q)_{j+\ell}}{(q)_j}.
\end{align*}
For the left hand side of \eqref{Andrews-(4.2)-transform}, it turns to be
\begin{align*}
\sum_{n=0}^{\infty} \frac{q^{n^2+n}}{(q)_{2n+1}} \cdot \frac{(q)_\infty}{2 \pi} \int_{-\pi}^{\pi} H_\ell(x|q) (-qe^{i\theta}, -qe^{-i\theta})_n (e^{2i \theta}, e^{-2i \theta} )_\infty \ d\theta,
\end{align*}
which directly completes the proof by using Corollary \ref{integral1-generalize}.
\end{proof}

When $\ell=0$ in the above theorem, we obtain the following double-sum identity of Rogers--Ramanujan type.

\begin{corollary}\label{Andrews-(4.2)-orth-t=0}
We have
\begin{align} \label{Andrews-4.2-i}
\sum_{n=0}^{\infty} \sum_{m=-\lfloor \frac{n+1}{2} \rfloor}^{\lfloor \frac{n}{2} \rfloor} \frac{(-1)^m \ q^{n(n+1)+m(5m+3)/2}}{(q)_{2n+1}} {2n+1 \brack n-2m} =\frac{(q^3, q^{10}, q^{13}; q^{13})_\infty}{(q)_\infty}.
\end{align}
\end{corollary}

\begin{proof}
By taking $\ell=0$ in Theorem \ref{Andrews-(4.2)-orth}, it becomes
\begin{align}
&\sum_{n=0}^{\infty} \frac{q^{n(n+1)}}{(q)_{2n+1}} \sum_{m=-\infty}^{\infty} (-1)^m q^{\binom{2m+1}{2}+ \binom{m+1}{2}}  \left({2n+1 \brack n-2m}- q^{2m+1} {2n+1 \brack n-2m-1} \right)  \nonumber \\
&\ =\frac{2}{(q)_\infty} \sum_{j=0}^{\infty} (-1)^j q^{3\binom{2j+1}{2} + \binom{j+1}{2}} (1-q^{6j+3}). \label{Andrews-(4.2)-orth-step1}
\end{align}
For the right hand side of  \eqref{Andrews-(4.2)-orth-step1}, by taking $j \mapsto -j-1$ into the second term, and applying  Jacobi's triple product identity \eqref{JTP}, it implies that
\begin{align*}
\frac{2}{(q)_\infty} \sum_{j=-\infty}^{\infty} (-1)^j q^{3\binom{2j+1}{2} + \binom{j+1}{2}} =\frac{2(q^3, q^{10}, q^{13}; q^{13})_\infty}{(q)_\infty},
\end{align*}
For the left hand side, letting $m \rightarrow -m-1$ in the second term, it turns to be
\begin{align*}
2\sum_{n=0}^{\infty} \sum_{m=-\lfloor \frac{n+1}{2} \rfloor}^{\lfloor \frac{n}{2} \rfloor} \frac{(-1)^m \ q^{n(n+1)+m(5m+3)/2}}{(q)_{n-2m} \ (q)_{n+2m+1}},
\end{align*}
which completes the proof.
\end{proof}

Note that \eqref{Andrews-4.2-i}
gives a double-sum expression  for Andrews--Gordon identity \eqref{Andrews-Gordon-identity} with $k=6$ and $i=3$, in which it takes the form of a quintuple sum.

From the identity due to Sun  \cite[(1.6)]{Sun23}
\begin{align} \label{Sun-1.6}
\sum_{n=0}^{\infty} \frac{q^{2n^2} \prod_{j=1}^n (1+2xq^{2j-1}+q^{4j-2})}{(q^2;q^2)_{2n}}
=\frac{1}{(q^2;q^2)_\infty} \sum_{n=0}^{\infty} q^{3n^2} \big( V_n(x)+V_{n-1}(x) \big),
\end{align}
we obtain the following triple sum identity with a free variable $\ell$.

\begin{theorem}\label{Sun-(1.6)-orth}
We have
\begin{align*}
&\sum_{n=0}^{\infty} \frac{q^{2n^2}}{(q^2; q^2)_{2n}} \sum_{m=-\infty}^{\infty} \sum_{k=0}^{\ell} (-1)^m q^{(2m+2k-\ell)^2+ \binom{m+1}{2}} {\ell \brack k} \nonumber\\
&\quad
\cdot \left({2n \brack n-2m-2k+\ell}_{q^2}
-q^{4(2m+2k-\ell+1)} {2n \brack n-2m-2k+\ell-2}_{q^2} \right) \\
&=\frac{2}{(q^2; q^2)_\infty} \sum_{j=0}^{\infty} (-1)^j q^{\binom{j+1}{2} + 3(2j+\ell)^2} (1-q^{12(2j+\ell+1)} ) (q^{j+1})_{\ell}.
\end{align*}
\end{theorem}

\begin{proof}
Substituting \eqref{V_n} into identity \eqref{Sun-1.6}, we obtain
\begin{align*}
\sum_{n=0}^{\infty} \frac{q^{2n^2} (-qe^{i\theta}, -qe^{-i\theta}; q^2)_n}{(q^2;q^2)_{2n}} &=\frac{1}{(q^2;q^2)_\infty} \sum_{n=0}^{\infty} q^{3n^2} \big( U_n(x)- U_{n-2}(x) \big)\\
&=\frac{1}{(q^2;q^2)_\infty} \sum_{n=0}^{\infty} q^{3n^2} (1-q^{12n+12}) U_n(x).
\end{align*}
Replacing $U_n(x)$ by \eqref{U_n-to-Hermite}, it gives
\begin{align*}
&\sum_{n=0}^{\infty} \frac{q^{2n^2} (-qe^{i\theta}, -qe^{-i\theta}; q^2)_n}{(q^2;q^2)_{2n}}\\
&\qquad
=\frac{1}{(q^2;q^2)_\infty} \sum_{n=0}^{\infty} q^{3n^2} (1-q^{12n+12}) \sum_{j=0}^{\lfloor \frac{n}{2} \rfloor} {n-j \brack j} (-1)^j q^{\binom{j+1}{2}} H_{n-2j} (x|q).
\end{align*}
Then by multiplying both hand sides by \eqref{multiply-term}, applying the orthogonality relation \eqref{Hermite-orth} and Lemma \ref{integral2-generalize}, the proof is completed.
\end{proof}

When $\ell=0$ in Theorem \ref{Sun-(1.6)-orth}, it reduces to the following double-sum identity of Rogers--Ramanujan type.

\begin{corollary}
We have
\begin{align*}
\sum_{n=0}^{\infty} \sum_{m=-\lfloor \frac{n}{2} \rfloor}^{\lfloor \frac{n}{2} \rfloor} \frac{(-1)^m \ q^{2n^2+m(9m+1)/2}} {(q^2; q^2)_{2n}} {2n \brack n-2m}_{q^2} =\frac{(q^{12}, q^{13}, q^{25}; q^{25})_\infty}{(q^2; q^2)_\infty}.
\end{align*}
\end{corollary}
Moreover, we refer to Chu and Wang's result  \cite[Cor. 38]{Cw12} with $l=6$,  which is a quintuple sum representation for the right hand side of the above identity.

From Andrews' identity \eqref{Andrews-(5.1)} on $W_n(x)$, we obtain the following result with a free variable $\ell$.

\begin{theorem}\label{Andrews-(5.1)-orth}
We have
\begin{align*}
&\sum_{n=0}^{\infty} \frac{q^{n(n+1)}}{(q)_{2n}} \sum_{m=-\infty}^{\infty} \sum_{k=0}^{\ell} (-1)^m q^{\binom{2m+2k-\ell}{2}+ \binom{m+1}{2}} {\ell \brack k}  \\
&\quad
\cdot \Bigg( {2n-1 \brack n-2m-2k+\ell} +q^{2m+2k-\ell} {2n-1 \brack n-2m-2k+\ell-1} -q^{2(2m+2k-\ell)+1}  \\
&\qquad\qquad
\cdot {2n-1 \brack n-2m-2k+\ell-2} -q^{3(2m+2k-\ell+1)} {2n-1 \brack n-2m-2k+\ell-3}\Bigg)\\
&=\frac{2}{(q)_\infty} \sum_{j=0}^{\infty} (-1)^j q^{\binom{j+1}{2} + (2j+\ell)(6j+3\ell+1)/2} (1-q^{4j+2\ell+1} +q^{6j+3\ell+2} -q^{10j+5\ell+5}) (q^{j+1})_{\ell}.
\end{align*}
\end{theorem}

\begin{proof}
By using formula \eqref{W_n}, identity \eqref{Andrews-(5.1)}  can be written as follows
\begin{align*}
&\sum_{n=0}^{\infty} \frac{q^{n^2+n} (-e^{i\theta}, -e^{-i\theta})_n}{(q)_{2n}}
=\frac{1}{(q)_\infty} \sum_{n=0}^{\infty} q^{n(3n+1)/2} (1-q^{2n+1} +q^{3n+2} -q^{5n+5})\ U_n(x).
\end{align*}
Expanding   $U_n(x)$ in terms of $H_n(x|q)$ by  using \eqref{U_n-to-Hermite},  multiplying both hand sides by \eqref{multiply-term}, and then  applying the orthogonality relation \eqref{Hermite-orth} and Corollary \ref{integral3-generalize}, the proof is completed.
\end{proof}

Taking $\ell=0$ in Theorem \ref{Andrews-(5.1)-orth}, we  get the identity as follows.

\begin{corollary}
We have
\begin{align}\label{Andrews-(5.1)-orth-simply}
&1+\sum_{n=1}^{\infty} \sum_{m=-\lfloor \frac{n-1}{2} \rfloor}^{\lfloor \frac{n}{2} \rfloor} \frac{(-1)^m \ q^{n(n+1)+m(5m-3)/2} \ (1+q^m)} {(q)_{2n}} {2n-1 \brack n-2m} \\
&\qquad\qquad\qquad\qquad
=\frac{1}{(q)_\infty}\big((q^5, q^{8}, q^{13}; q^{13})_\infty)-q(q, q^{12}, q^{13}; q^{13})_\infty \big).\nonumber
\end{align}
\end{corollary}

\begin{proof}
When $\ell=0$, the right hand side of  Theorem \ref{Andrews-(5.1)-orth} turns to be
\begin{align*}
\frac{2}{(q)_\infty} \sum_{j=0}^{\infty} (-1)^j q^{\binom{j+1}{2} +j(6j+1)} (1 -q^{4j+1} +q^{6j+2} -q^{10j+5}).
\end{align*}
Setting  $j \mapsto -j-1$ into  the third and forth summands, it becomes
\begin{align*}
\frac{2}{(q)_\infty} \sum_{j=-\infty}^{\infty} (-1)^j q^{\binom{j+1}{2} +j(6j+1)} (1 -q^{4j+1}).
\end{align*}
By using Jacobi's triple product identity \eqref{JTP}, it leads to the right hand side.  For the left hand side, by applying $m\mapsto -m-1$ to the third and fourth summands and then simplifying, we complete the proof.
\end{proof}

Noting that by adding   \eqref{Yao-(3.60)-orth-simply} and \eqref{Andrews-(5.1)-orth-simply} together, it implies that
 \begin{align*}
&2+\sum_{n=1}^{\infty} \sum_{m=-\lfloor \frac{n-1}{2} \rfloor}^{\lfloor \frac{n}{2} \rfloor} \frac{(-1)^m q^{n^2 +m(5m-3)/2} (1+q^n) (1+q^m)} {(q)_{2n}} {2n-1 \brack n-2m} \\
&\qquad\qquad\qquad\qquad\qquad
=\frac{1}{(q)_\infty} \big( (q^5, q^{8}, q^{13}; q^{13})_\infty +(q^6, q^{7}, q^{13}; q^{13})_\infty \big).
\end{align*}

From  the identity given by Yao \cite[(3.55)]{Yao25}
\begin{align*}
&1+\sum_{n=1}^{\infty} \frac{q^{n^2+2n} (-q;q^2)_n \prod_{j=0}^{n-1} (1+2xq^{2j}+q^{4j})}{(q^2;q^2)_{2n}}\\
&\qquad\qquad
=\frac{(q^2;q^2)_\infty}{(q)_\infty(q^4;q^4)_\infty} \sum_{n=0}^{\infty} q^{2n^2+n} (1-q^{2n+1}) \big( U_n(x)+U_{n-1}(x) \big),
\end{align*}
we obtain the following triple sum identity with a free variable $\ell$.

\begin{theorem}\label{Yao-(3.55)-orth}
We have
\begin{align*}
&\sum_{n=0}^{\infty} \frac{q^{n(n+2)} (-q; q^2)_{n}}{(q^2; q^2)_{2n}} \sum_{m=-\infty}^{\infty} \sum_{k=0}^{\ell} (-1)^m q^{2\binom{2m+2k-\ell}{2}+ \binom{m+1}{2}} {\ell \brack k}  \\
&\quad
\cdot \Bigg({2n-1 \brack n-2m-2k+\ell}_{q^2} +q^{2(2m+2k-\ell)} {2n-1 \brack n-2m-2k+\ell-1}_{q^2} -q^{4(2m+2k-\ell)+2}  \\
&\qquad\qquad
\cdot {2n-1 \brack n-2m-2k+\ell-2}_{q^2} -q^{6(2m+2k-\ell+1)} {2n-1 \brack n-2m-2k+\ell-3}_{q^2}\Bigg) \\
&=\frac{2(q^2; q^2)_\infty}{(q)_\infty (q^4; q^4)_\infty} \sum_{j=0}^{\infty} (-1)^j q^{\binom{j+1}{2} + (2j+\ell)(4j+2\ell+1)} \\
&\qquad\qquad\qquad\qquad
\cdot \big( 1 -q^{2(2j+\ell)+1} +q^{4(2j+\ell)+3} -q^{6(2j+\ell)+6} \big) (q^{j+1})_{\ell}.
\end{align*}
\end{theorem}

\begin{proof}
The above Yao's identity  can be rewritten  as
\begin{align*}
&\sum_{n=0}^{\infty} \frac{q^{n^2+2n} (-q;q^2)_n (-e^{i\theta}, -e^{-i\theta} ;q^2)_n}{(q^2;q^2)_{2n}}\\
&\qquad\qquad
=\frac{(q^2;q^2)_\infty}{(q)_\infty(q^4;q^4)_\infty} \sum_{n=0}^{\infty} q^{2n^2+n} (1-q^{2n+1}+q^{4n+3}-q^{6n+6})\ U_n(x).
\end{align*}
By using the expansion formula \eqref{U_n-to-Hermite}, multiplying both sides by \eqref{multiply-term}, then applying the orthogonality relation \eqref{Hermite-orth} and Corollary \ref{integral4-generalize}, the proof is complete.
\end{proof}

Taking $\ell=0$ in Theorem \ref{Yao-(3.55)-orth}, it yields the following identity.

\begin{corollary}
We have
\begin{align*}
&1 +\sum_{n=1}^{\infty} \sum_{m=-\lfloor \frac{n-1}{2} \rfloor}^{\lfloor \frac{n}{2} \rfloor} \frac{(-1)^m q^{n(n+2)+m(9m-5)/2} (1+q^m)} {(q; q^2)_{n} (q^4; q^4)_{n}} {2n-1 \brack n-2m}_{q^2} \\
&\qquad\qquad\qquad\qquad
=\frac{(q^2; q^4)_\infty}{(q)_\infty} \big( (q^6, q^{11}, q^{17}; q^{17})_\infty- q(q^2, q^{15}, q^{17}; q^{17})_\infty \big).
\end{align*}
\end{corollary}

Finally, we remark that if $\ell$ is taken to be a special value greater than $0$ in the theorems obtained in this section, it will lead to a series of identities related to partial theta functions, and more details of partial theta functions can be seen in \cite{And81, AB09}.

\section*{Acknowledgments}
	This work is supported by the National Natural Science Foundation of China (No. 12571351, 12071235), Tianjin Natural Science Foundation (No. 24JCZD\\JC01390), the Research Program on Graduate Education Reform in Tianjin Higher Education Institutions (Grant No. TJYG035) and the Fundamental Research Funds for the Central Universities of China.

\end{document}